\documentclass[a4paper,12pt,twoside]{article}
\usepackage{amsmath,amssymb,amsthm,textcomp}
\usepackage[english]{babel}

\newcounter{ass}

\RequirePackage{amsmath}
\RequirePackage{amssymb}
\def\O{\Omega}

\def\s{\sigma}

\def\d{\delta}

\def\l{\lambda}
\def\L{\Lambda}
\def\eps{\varepsilon}
\def\a{\alpha}
\def\b{\beta}
\def\bar{\overline}
\def\G{\Gamma}
\def\g{\gamma}
\def\disp{\displaystyle}

\def\r{\rho}

\def\dpa{\partial}
\def\t{\tau}
\def\f{\varphi}

\def\Th{\Theta}
\def\E{{\bf E}}
\def\T{\bf T^d}
\def\Prb{{\bf P}}
\def\tl{\widetilde}
\def\qs{\forall\;}

\def\td{\longrightarrow}
\def\Td{\Longrightarrow}
\def\tdeps{\xrightarrow{\eps\td 0}}
\def\N{\mathbb{N}}
\def\Z{\mathbb{Z}}
\def\R{\mathbb{R}}

\newtheorem{theo1}{Theorem}[section]
\newtheorem{defini}[theo1]{Definition}
\newtheorem{propo}[theo1]{Proposition}
\newtheorem{lem1}[theo1]{Lemma}
\newtheorem{corollary}[theo1]{Corollary}

\numberwithin{equation}{section}
\begin{document}
\title{Homogenization of a periodic semilinear  elliptic   degenerate PDE}
\author{\'Etienne Pardoux$^{(a)}$ \and Ahmadou Bamba Sow$^{(b)}$}
\date{\today}
\maketitle \noindent{\footnotesize
$^{(a)}$ CMI, LATP-UMR 6632, Universit\'e de Provence, 39 rue F.
Joliot Curie, Marseille cedex 13, FRANCE. \\ {\tt email :
pardoux@cmi.univ-mrs.fr}\\
 $^{(b)}$ LERSTAD, UFR S.A.T,
Universit\'e Gaston Berger, BP 234, Saint-Louis, SENEGAL. \\email :
{\tt
ahmadou-bamba.sow@ugb.edu.sn} }
\begin{abstract}
{\small In this paper  a  semilinear  elliptic PDE  with rapidly
oscillating  coefficients is  homogenized. The novetly  of our
result  lies  in the fact that   we allow the second order part  of
the  differential
operator  to be degenerate in some  portion   of  $\R^d$.\\
Our  fully probabilistic method is based on the  connection between   PDEs and   BSDEs with  random  terminal
time and   the  weak convergence  of a class of   diffusion  processes. \vspace{0.5cm} }
\end{abstract}
%
%
\section{Introduction}
 The theory of homogenization tries to understand what equations should be used at a macroscopic level, in order to approximate the behavior of  physical phenomena described at a microscopic level by equations  with highly oscillatory coefficients. This theory has motivated the development of various notions of weak convergence in analysis, see in particular Tartar \cite{Tar}. One way to understand such convergence, at least in the case of linear or certain semilinear equations with periodic coefficients is based on a probabilistic interpretation of the equation, see among others Freidlin \cite{Frei1}, \cite{Frei2} in the linear case,  Briand  and  Hu  \cite{Bri-Hu},  Pardoux \cite{Par1}
 and Delarue \cite{Del} in the semilinear case. The last  three  papers exploit the connection between BSDEs and semilinear PDEs, see Pardoux and Rascanu \cite{Par-Rascanu}.

Recently Hairer and Pardoux \cite{Hai-Par}  have generalized the probabilistic approach to the homogenization of linear second--order periodic PDEs with periodic coefficients to systems where the matrix of second order
coefficients can be allowed to degenerate or even vanish on an open subset of $\R^d$. Those results have been extended to semilinear parabolic PDEs in Sow, Rhodes and Pardoux \cite{Sow-Rho-Par}. The aim of the present paper is to treat a class of semilinear elliptic PDEs, whose matrix of second order ocefficients is allowed to vanish in part of $\R^d$.

More precisely, we will study the homogenization of the elliptic Dirichlet boundary value problem
in the smooth
bounded domain $G \subset  \R^d$ :
\begin{equation}
\left\{
\begin{aligned}
\disp L_\eps \, u^\eps(x)   &+  f \left(\frac{x}{\eps}, x ,  \,  u^\eps(x)
, \, \dpa_x u^\eps(x) \s\left(\frac{x}{\eps}\right)\right)  =  0,
   \quad x  \in G, \\
u^\eps(x) &=  g(x), \quad   x  \in \dpa  G, \\
\end{aligned}
\right.
\end{equation}
where   the  second  order  differential   operator  with  rapidly  oscillating  coefficients,  $L_\eps$  is given  by
\begin{equation}\label{L-epsintro}
L_\eps(\cdot) =  \frac{1}{2} \sum_{i, j = 1}^d \,
a_{ij}\left(\frac{x}{\eps}\right) \, \dpa^2 _{ x_i x_j}\cdot
+ \sum_{i = 1}^d \,  \bigg[\frac{1}{\eps} b_i \left(\frac{x}{\eps}\right) +
c_i \left(\frac{x}{\eps}\right) \bigg]\, \dpa_{x_i}\cdot,
\end{equation}
 $a, b, c$  being  periodic   functions ($a   =  \s \s^*$  for
some  periodic  function  $\s$). \\

 The paper   is organized as
follows. Section 2 contains our main assumptions, some  preliminary
results including ergodicity. In
  section  3 we  prove our main   theorem, while the proof of several technical results is delayed until section 4.
\section{Diffusions with periodic coefficients}
In all what follows, we assume given a complete stochastic basis  $(\O, \,   {\cal F}, \,   ({\cal F}_t)_{0\le t \le T}, \,    \Prb)$,
 where  the filtration $  ({\cal F}_t)_{0\le t \le T}$  is generated by a  $d$-dimensional Brownian motion  $(B_t)_{0\le t \le T}$
 and       the  continuous functions
$$ b,  \;  c :  \R^d  \td  \R^d ,  \qquad  \s :  \R^d  \td  \R^d  \times  \R^d
$$
are  periodic of period  1  in each direction  of  $\R^d$. \\
Given $\eps > 0, \; x \in \R^d$, let $(X_t^{x,\eps})_{t\ge 0}$ (which will be mostly written $X_t^{\eps}$) denote the solution
of the stochastic differential equation
\begin{eqnarray}\label{diffusion-X}
\qs  t  \ge  0, \quad   X_t^\eps = x  + \int_0^t \left[\frac{1}{\eps} b
\left(\frac{X_s^\eps}{\eps}\right) +  c
 \left(\frac{X_s^\eps}{\eps}\right)\right] \, ds  + \sum_{j = 1}^d \int_0^t \s_j\left(\frac{X_s^\eps}{\eps}\right) \,
d B_s^j
\label{dif}
\end{eqnarray}
and
\begin{equation*}
L_\eps(\cdot) = \frac{1}{2} \sum_{i, j = 1}^d \,
a_{ij}\left(\frac{x}{\eps}\right) \, \dpa^2 _{ x_i x_j}\cdot
+ \sum_{i = 1}^d \,  \bigg[\frac{1}{\eps} b_i \left(\frac{x}{\eps}\right) +
c_i \left(\frac{x}{\eps}\right) \bigg]\, \dpa_{x_i}\cdot
\end{equation*}
its  infinitesimal generator where  for every  $x\in  \R^d,  \quad
a(x)  =  \s(x) \,  \s^*(x)$. Considering  the   processes  $ ({\tl
X}_t^\eps)_{t \ge 0}$ and $ ({\bar  X}_t^\eps)_{t  \ge 0}$   defined
by
$$ \qs  t  \ge 0, \quad   {\tl X}_t^\eps = \frac{1}{\eps}\, X_{\eps^2 t}^\eps \qquad;  \quad      {\bar  X}_t^\eps =
\frac{X_t^\eps}{\eps} = {\tl X}_{t/\eps^2}^\eps, $$ then there
exists a standard $d-$dimensional Brownian motion $(B_t)_{t \ge 0}$
depending on $\eps$ (in fact  $B_s^\eps = \frac{1}{\eps} B_{\eps^2
s}$ and we forget that dependence), such that  for all $t\ge0$,
\begin{equation}
 {\tl X}_t^\eps = \frac{x}{\eps} + \int_0^t [
b({\tl X}_s^\eps) + \eps \, c ({\tl X}_s^\eps)]ds + \sum_{j =
1}^d \int_0^t \s_j( {\tl X}_s^\eps) \, d B_s^j .\label{tilde X}
\end{equation}

We consider the Markov process    $({\tl X}_t^\eps)_{t  \ge 0} $
solution  of \eqref{tilde X} as taking values in the $d$ dimensional
torus ${\bf T^d} = \R^d/ \Z^d$ and denote by $p^\eps (t, x, A)$ its transition
probability. We shall write  $p(t, x,  A) $  for $p_0(t, x,  A) $. 
 
We  will  also consider  the same
equation  starting  from  $x$  but  without  the term  $\eps c$,
namely
\begin{equation}
\qs  t\ge 0,  \qquad {\tl X}_t^x  = \,  x    + \int_0^t b ({\tl
X}_s^x) \,  \, d s  + \sum_{j = 1}^d \int_0^t \s_j({\tl X}_s^x) \, d
B_s^j \label{X-x}
\end{equation}
and let  $ (J_t^{x})_{t \ge 0} $ denote the  Jacobian  of  the  stochastic
flow associated  to   $(\tl X_t^{x})_{t \ge 0}$,  that  is  the
$d  \times d $ matrix valued stochastic   process  solving
\begin{equation}
d J_t^{x}  =  D b(\tl X_t^{x}) \,  J_t^{x}  \, d t + \sum_{j = 1}^d
D \s_j(\tl X_t^{x}) \, J_t^{x}  \; d B_t^j, \qquad J_0^{x} = I.
\end{equation}
Moreover  to  the   stochastic  differential  equation satisfied by
$(\tl X_t^{x})_{t  \ge  0}$, having  in mind  Stroock-Varadhan's
support  theorem, we  associate  the following  controlled ODE
(where we use the convention  of  summation over repeated indices).
For each $x \in   {\T},  \;  \;   u   \in L_{loc}^2 (\R_+, \,
\R^d)$,  $(z_u^ {x, \eps} (t),  \,  t\ge  0)$  denotes the
solution  of
\begin{equation}\label{eq-z}
\left\{
\begin{aligned}
\dfrac{d  z_i}{dt}(t)  &=  (b_i + \eps c_i) (z(t))  - \dfrac{1}{2} \bigg(\dpa_{x_k}  \s_{ij} \s_{kj} \bigg) (z(t))   + \s_{ij}(z(t)) u_j(t); \\
z(0)  &= x.
\end{aligned}
\right. 
\end{equation}
\subsection{Assumptions and  preliminary results}
Let  us  recall  the  following
\begin{defini}
Consider   $b$  and      the  columns  vectors  $\s_j$ of   $\s$ as
vector fields  on  the  torus $\T$. We will  say  that the strong
H\"ormander condition holds  at  some  point  $x  \in  \T$ if  the
Lie algebra generated  by   $ \{ \s_j(x)\}_{1 \le  j \le d} $  spans
the  whole tangent
space  of  $\T$  at $x$. \\
We   furthermore  say  that   the  parabolic   H\"ormander condition
holds  at   $x$,  if  the   Lie  algebra generated  by the
$(d+1)$--dimensional vectors $(b, 1) \cup \{(\s_j, 0)\}_{1 \le j \le d}$
spans the whole space  $\R^{d + 1}$  at   $(x, 0)   \in  \T \times  \R$.
\end{defini}
We say  that  the drift  and   the  diffusion  coefficients satisfy
the assumption    {\bf (H1)} ((H1.1) to (H1.5))  if the following holds
\\ \\
{\bf (H1.1)} $\s,  \; b$  and  $c$ are of  class  $\mathcal{C}^\infty$.  \\ \\
{\bf (H1.2)} There  exists  a  non empty,  open  and  connected
subset  $U$  of  $\T$  on  which  the  strong   H\"ormander
conditions  holds. Futhermore,  there  exists  $t_0$  and  $\eps_0$
such  that
$$
\qs   x  \in  {\T},  \quad   0 \le \eps \le    \eps_0,  \qquad
\inf_{u  \in  L^2(0, t_0, \R^d)} \{|| u ||_{L^2}\;  ;   z_u^{x,
\eps} (t_0)   \in  U\} \quad  <  \;  \infty.
$$
{\bf (H1.3)} If $V$  denotes  the  subset   of  $\T$  where the
parabolic   H\"ormander condition holds,    $\t_V^x$   the first
hitting  time of  $V$  by  the  process  $\{\tl  X_t^x\}$,  then
$$
\inf_{t  >  0}\sup_{x  \in  {\T}}  \, \E  (|J_t^x|,  \;  \{\t_V^x \;
\ge  t\})  \;  <   \,  1.
$$

\noindent It   is  not difficult   to  verify that under {\bf (H1.1)} and {\bf
(H1.2)}  the  following  Doeblin condition is satisfied~: there
exists $t_1
> 0, \; 0 < \eps_1 \, <  \eps_0,  \;  \b  > 0$  and $\nu$  a
probability measure on  $\T$  which is  absolutely continuous with
respect to the Lebesgue  measure, s.t.  for  all $0 <  \eps <
\eps_1, \; x \in {\T}$, $A$  a  Borel  subset   of $\T$,
$$  p_\eps(t_1; x,  A) \;  \ge  \b  \,  \nu(A).$$
This ensures existence   and  uniqueness   of  a unique invariant
measure  $\mu_\eps$ of $({\tl X}_t^\eps)_{t \ge 0} $ (we shall write
$\mu$  for $\mu_0$)  and the following facts   (see \cite{Hai-Par})
\begin{lem1}[The spectral gap]
There exists $C, \r > 0$ such that for all $ 0\leq \eps\leq 1, \; t \ge
0 $  and  $ f\in L^\infty(\T)$,
$$
\bigg\vert \E f({\tl X}_t^\eps )  - \int_{\T} f(x) \, \mu_\eps (d
x)\bigg \vert \leq \; C \|f\|_\infty \, e^{-\r t}.
$$
\label{spectral}
\end{lem1}
%
%
\begin{lem1}
The  following  holds $$\mu_\eps \tdeps  \;  \mu,   \quad
\mbox{weakly}.$$
\end{lem1}
As  a  consequence    we  have the   following  sort  of  ergodic  theorem
\begin{corollary}\label{ergodique}
Let  $f \in  L^\infty({\T})$. Then   for  any   $t\ge 0$,
$$\int_0^t  f(\bar X_s^{\eps, x})   d s   \to    \;   t  \int_{\T}  f(y) \mu(dy) $$
\end{corollary}
We finally assume that 
\\
{\bf (H1.4)} The crucial centering condition is
satisfied :$\quad \disp \int_{\T} \, b(x) \, \mu(d x)  = 0 $.
\subsection{The Poisson equation}
Let   us  consider   the  infinitesimal   generator   $L$    of  the
$\T-$valued  diffusion  process  $(\tl  X_t^x)_{t  \ge  0}$  given
by
\begin{equation}
L  =    \dfrac{1}{2}\sum_{i,j=1}^d  (\s\s^*)_{ij}(x)  \dpa^2 _{ x_i
 x_j}  +  \sum_{i = 1}^d b_i(x)   \dpa_{ x_i}
\end{equation}
and  ${\cal  P}_t$  the  semigroup  generated    by  $(\tl  X_t^x)_{t  \ge  0}$.  \\
For   some  functions  $f \in \mathcal{C}^1(\T)$ satisfying  the centering
condition
\begin{equation}
 \int_{\T} f(x) \, \mu (d x) = 0,
\label{cond-Poisson}
\end{equation}
we   want to solve the PDE
\begin{equation}
L \widehat f (x)  + f(x) = 0 , \;\; \; x \in \T. \label{Poisson}
\end{equation}
This   will be essential in  order   to  get  rid  of  the   terms  depending  on $\eps^{-1}$
in the perturbed  equations. For  this  purpose  we recall the
following result given  in  \cite[Lemma 2.6]{Hai-Par} which will be
useful in the sequel :
\begin{lem1} \label{lemmeC1} Under   {\bf (H1.1)--(H1.3)},   ${\cal  P}_t$ maps    $\mathcal{C}^1(\T)$   into  itself   and   there
exists two  positive  constants  $K  > 0$  and    $\r   >  0$  such
that for every $f \in \mathcal{C}^1(\T)$  satisfying   \eqref{cond-Poisson}
and for every $t \ge  0$,  we  have
\begin{equation}
\|  {\cal  P}_t \,  f \|_{\mathcal{C}^1(\T)}  \,  \le   K e^{-\r  t} \|f
\|_{\mathcal{C}^1(\T)}.
\end{equation}
\end{lem1}
It follows from  Lemma  \ref{lemmeC1} the
\begin{lem1}   Under   {\bf (H1.1)--(H1.3)},
  if  $f  \in  \mathcal{C}^1(\T)$ satisfies \eqref{cond-Poisson}, then  the function  $\widehat f$ defined  by
$$  \widehat f(x)  = \;   \int_0^{+ \infty} \E_x[f(\tl
X_t)]\, dt, \quad  x \in  {\T}, $$  belongs to $\mathcal{C}^1(\T)$  and is the
unique weak sense  solution
 of equation \eqref{Poisson} which is centered with respect to
 $\mu$.
\end{lem1}
For  the  notion  of  weak sense  solution to  \eqref{Poisson}, see
\cite{Par-Ver3}.
Under the previous assumptions, for $i=1,\dots,d, $ we can consider the
following Poisson equation  on the torus $\T$:
\begin{equation}\label{poissone}
  L \widehat  b_i (\cdot)  +  b_i(\cdot)= 0.
\end{equation}
Thanks  to  Lemma \ref{lemmeC1}, for  any  $i  =  1,  \dots ,d$, the
function $\widehat b_i$ solution   of \eqref{poissone}   belongs   to
$\mathcal{C}^1(\T)$   and  is   given for any $ x \in {\T} $ by
\begin{equation}\label{explpoisson}
   \widehat  b_i
(x)=\int_0^{+\infty} \E_x [b_i (\tl X_t)]\, dt.
\end{equation}
Let    us  consider    the  constant  coefficients  $A$   and $C$
given
 by (with   $ \L (x) = (I + \dpa_x \widehat b) (x) \s(x)$)
\begin{equation*}
A  = \int_{\T} (\L \L^*) (x) \, \mu(d x)
 ;  \quad
C  = \int_{\T} ( I  + \dpa_x \widehat b)  c  (x) \, \mu(d x)
\end{equation*}
and    the   diffusion $(X^x_t)_{t\ge 0}$   given   by
\begin{equation}
\qs t \ge 0, \quad X^x_t =    x     + C t + \; A^{1/2} B_t.
\end{equation}
We state   the   following  crucial  condition   \\
{\bf (H1.5)}  The  matrix   $A$ is    positive   definite.  \\

\noindent\textbf{Remark :} \emph{
Necessary and   sufficient  condition   for  (H1.5)  to    hold  are  given  in \cite{Hai-Par} in  terms  of   the   diffusion  $(\tl X_t)_{t\ge 0}$ and  the support  of  its  invariant  measure.}  \\ \\
Recall the subset $G\subset\R^d$ from the Introduction.
Define  the  stopping times   $\t_x^\eps = \inf\{t \ge 0,\,  X_t^{\eps,x} \notin
\overline{G}\} $ and   $\t_x = \inf\{t \ge 0,\,  X_t^{x} \notin
\overline{G}\} $ (the   subscript $x$   will   be  often  omitted    for   notational  simplicity). Note    that (H1.5)  implies  that    $\t_x = 0, \, a.s.$ for  all  $x\in  \dpa G.$ \\

\noindent We  have    the following result established  in  \cite[Theorem 3.1]{Hai-Par}
\begin{propo}\label{hairer-Pard}
  Under  assumptions  {\bf
(H1)}, the   following  weak  convergence    holds
\begin{align*}
(X^{x, \eps} ,  \t_x^\eps)  \Td  (X^{x} ,  \t_x) \quad   \text{in} \quad    {\cal  C}(\R_+, \bar{G})\times \R_+.
\end{align*}
\end{propo}

\noindent We are    now  in  position  to  study our  main   subject.
\section{Homogenization of an elliptic  PDE}
For each  $\eps   > 0$,   we consider   the  elliptic  PDE    with
Dirichlet boundary condition
\begin{eqnarray}
\begin{cases}
\disp  L_\eps \, u^\eps(x)  \;  +    f \left(\frac{x}{\eps},  x , \; u^\eps(x), \dpa_x u^\eps(x) \s\left(\frac{x}{\eps}\right)\right)  =  0, \quad  x  \in  G, \\
u^\eps(x) =  g(x), \quad  x  \in  \dpa G
\end{cases}
\label{elliptic3}
\end{eqnarray}
where   $G \subset \R^d$  is   a smooth  bounded domain whose
boundary is of class  $\mathcal{C}^2$,  $g \in  {\cal C}^2(\dpa G) $ and
$f : \R^d  \times  \R^d \times \R  \times \R^d \td \R $ is   continuous  and
satisfies the following assumptions {\bf (H2)} (in what follows, the reader  should 
keep  in  mind   that  $y$   stands    for  $u^\eps$   and  $z$  for  $\dpa_x u^\eps$) : \\
{\bf (H2.1)} $f$  is  periodic of period one in each direction
of  $\R^d$ in  the  first  argument. \\
{\bf (H2.2)} There  exists   two constants $\mu  < 0$  and $K >
 0$  such that for every $x \in  {\T}, \;  \; (\tl x, \tl x^\prime) \in \R^d \times  \R^d$, $(y, y^\prime)  \in  \R^2$   and  $ (z, z^\prime)  \in  (\R^d)^2$,
\begin{align*}
(y   -     y^\prime)   ( f(x,  \tl x,   y, z)  - f(x,  \tl x, y^\prime, z) )&\le    \mu(y   -     y^\prime)^2  \\
|f(x,  \tl x,   y, z)  - f(x,  \tl x,   y, z^\prime)|  &\le
 K  \,  ||z-z^\prime||  \\
|f(x,  \tl x,   y, z) |  &\le   K(1  +  |y|  +   |z|)
\end{align*}
{\bf (H2.3)} There  exists  a    function  $\f\in  \mathcal{C}(\R_+,  \R_+)$  satisfying  $\f(0) = 0$  such  that  for every $x \in  {\T}, \;  \; (\tl x, \tl x^\prime) \in \R^d \times  \R^d$, $(y, z)  \in  \R\times\R^d$,
$$  |f(x,  \tl x,   y, z)  - f(x,  \tl x^\prime,   y, z)|   \le  \f(|\tl x -  \tl x^\prime|(1 + |z|)). $$
{\bf (H2.4)} There exists a constant $\l \neq 0$ such that  $\l > 2 \mu  + K^2 $  and
\begin{equation}
\sup_{x\in \bar G} \sup_{\eps > 0} \, \E_x   \, e^{\l \, \t^\eps}  \;  < \infty. \label{cond-exp}
\end{equation}
{\bf (H2.5)} For  every   $\eps  >0$,   the   set $\G^\eps = \{x \in
\dpa G : \Prb(\tau_x^\eps
  > 0)  =  0\}$ is  closed. \\

\noindent Let us consider the second order differential  operator
\begin{equation}
\bar L     = \frac{1}{2} \sum_{i, j= 1}^d\,  A_{ij} \dpa_{x_i x_j}^2
 + \sum_{i = 1}^d \, C_{i}\,  \dpa_{x_i}.
\end{equation}
We  are  interested    in   the  elliptic    PDE
\begin{eqnarray}
\begin{cases}
\bar  Lu(x)  +    \overline{f}(x,  \,  u(x), \, \dpa_x u(x))  = 0 , \;\;\;\; \; x  \in  G, \\
u(x) =  g(x),  \;\;\; \;  x \in \dpa  G.
\end{cases}
\label{lim-elliptic3}
\end{eqnarray}
where    the function $\overline{f}$ is   given  by  (recall that $
\L (x) = (I + \dpa_x \hat b) (x) \s(x)$)
\begin{equation*}
\qs  (\tl x,  \,  y, \, z)   \in   \R^d  \times  \R \times  \R^d,
\;\qquad  \overline{f} (\tl x, \, y, \, z)  =  \int_{\T}   f(x, \tl
x, \, y, \L (x)  z)\, \mu (dx).
\end{equation*}
It is  easy  to  see  that    $\overline{f}$  is jointly  continuous   and  satisfies   assumption  (H2.2). So   using the       BSDE  with  random   terminal time
 $$ Y_{t\wedge\t}^x=    g(X_\t^x)  +   \int_{t\wedge \t}^\t   \overline{f}(X_r^x,  Y_r^x,  Z_r^x)   d r    -  \int_{t\wedge \t}^\t Z_r^x d B_r, \quad   \;   t\ge 0,    $$
we  deduce       thanks  to   \cite[Corollary 6.96]{Par-Rascanu},  that   under   our  standing  assumptions,    $u(x)   =  Y_0^x  \in  \mathcal{C}(\overline{G})$  and    is  the    unique   viscosity  solution  of  \eqref{lim-elliptic3}.\\
We  now  formulate our  main  result :
\begin{theo1}\label{main-result} Under   assumptions {\bf (H1)} and  {\bf (H2)},  for all $x \in  G$,  we have
$$ u^\eps (x)  \xrightarrow[]{\eps  \td 0}  \;  u(x).  $$
\end{theo1}

\noindent Before   proving  Theorem \ref{main-result},   let  us   establish
the following technical result to     face  the   lack   of   smoothness   of   $u$ (whose proof  is  given   in  section \ref{section-proofs}).

\begin{propo}\label{regul-u}
Assume  that  (H1)   and  (H2) are  in  force. Then     there
exists  a   sequence    of   functions    $(u^n)_{n\ge 1}  \subset
\mathcal{C}^\infty(\R^d)$ satisfying
: \\
(i) There  exists  a  constant
  $\G>0$  such  that
  $$ \qs   n\ge  1,  \quad  \qs   x  \in  G,  \quad |\dpa_x  u^n| +  |u^n(x)| \le    \G.$$
 (ii) $\disp\sup_{x \in  \overline{G}}\left[|u^n(x)-u(x)|+|\partial u^n(x)-\partial u(x)|\right]\to0,  \quad   \text{as} \quad   n  \to \infty.$
  \\
(iii) The sequence    $(u^n)_{n\ge 1}$ satisfies
$\disp  \bar  Lu^n(x)  +    \overline{f}(x,  \,  u^n(x), \, \dpa_x u^n(x)) \to  0
   \; \; \text{as} \; \;   n  \to \infty$  uniformly  on
$G^\d  =   \{x\in    G  :   d(x,   G^c)  \ge \d\}$,  for    any  $\d>0$. \\
(iv) For  any $p\ge1,  \quad  {\bar  L} u^n  +  {\bar  f}(\cdot, u^n(\cdot), \dpa_x u^n(\cdot))$ is  uniformly  bounded  in  $L^p(G)$.
\end{propo}

\subsection{The   homogenization  property}

 Our approach to   prove  our  main  result  is purely
probabilistic and is based on BSDE techniques. The strategy consists
in introducing the unique pair $(Y_t^{\eps, x}, \; Z_t^{\eps,
x})_{0\le t \le \t^\eps}$ of ${\cal F}_t$-progressively measurable
processes solution of  the BDSE with random terminal time
\begin{eqnarray}
 \qs t \ge 0, \quad   Y_{t\wedge\tau^\eps}^{\eps, x}  =   g (X_{\t^\eps }^{\eps, x})    + \disp \int_{{t\wedge \t^\eps }}^{\t^\eps} \,
   f(\bar X_r^{\eps, x}, X_r^{\eps, x}, \; \;   Y_r^{\eps, x}, \, Z_r^{\eps, x}) \, dr   -
   \int_{{t\wedge \t^\eps }}^{\t^\eps} \, Z_r^{\eps, x} \, d B_r
\label{Y-eps}
\end{eqnarray}
satisfying
\begin{equation} \label{integrabilite} \E \bigg( \sup_{0
\le t \le \t^\eps} \, e^{\l t} |Y_t^{\eps,x}|^2 + \int_{0}^{\t^\eps}
\, \, e^{\l t}\, |Z_t^{\eps, x}|^2 \, d t \bigg) \; \; <  \; \infty.
\end{equation}
It is well known  (see \cite[Theorem 5.3]{Par3}) that
the function  $u^\eps (x)  =  Y_0^{\eps, x}$  is a viscosity
solution of  \eqref{elliptic3}.

  Let  us  consider  the   process  $M_t^\eps
= -  \int_0^t Z_r^{\eps,x} d B_r, \;  t\ge0$.  We intend to study
the tightness property of the pair of processes $(Y_{\cdot}^\eps,
M_{\cdot}^\eps)$ indexed by $\eps
> 0$  in the space ${\cal D}(\R_+;\R^{d+1}) $ (the space of right
continuous functions having left limits) equipped with the
$S$--topology  of Jakubowski (see \cite{MZH} for further details). \\
For  this end it suffices    to   establish   this   result
on   the interval  $[0, T]$  for every  $T>0$.\\
Let  us recall  that  the sequence of quasi-martingales $\{U_s^n; \,
0 \le s \le T\}$ defined on  the filtered probability space $\{\O;
{\cal F}, ({\cal F}_s)_{0\le s \le T}, \; \Prb\}$ is tight whenever
$$
\sup_n [\sup_{0\le s \le T}  \E |U_s^n|  \;  + \;  CV_T^0 (U^n)]\; <
\infty,
$$
where  $CV_T^0(U^n)$,  the so-called "conditional  variation  of
$U^n$ on $[0, \, T]$", is defined as
$$ CV_T^0 (U^n) = \sup  \E\bigg(\sum_{i = 1} |\E (U_{t_{i +1}}^n   -
U_{t_i}^n \, \mid {\cal F}_{t_i})|\bigg)
$$ and the supremum is taken over all  partitions of the  interval
 $[0, T]$.

 We   claim that (the proof    is   given   in  section \ref{section-proofs})
\begin{propo}
\label{lem:meyer} There exists a positive   constant
$C_{\ref{lem:meyer}}\,
> 0$ such that
\begin{align*}
\sup_{\eps >  0} \,     \E \Big[\sup_{0\le r\le\tau^\eps}e^{\l r}|Y_r^\eps|^2\,
+ \int_0^{\t^\eps} e^{\l r}|Z_r^\eps|^2\, d
 s\Big]
\; \le  C_{\ref{lem:meyer}}.
\end{align*}
\end{propo}

As  a  consequence,   we   have
\begin{corollary}\label{tightness}
For  any  $T>0$,  the  family  of  processes  $(Y_{\cdot}^\eps,
M_{\cdot}^\eps)$ indexed  by   $\eps>0$   is  $\Prb-$ tight   as
elements  of ${\cal D}([0, \, T], \R^{d+1})$, equipped  with   the
$S$--topology of Jakubowski.
\end{corollary}
To deal   with   the  highly    oscillating  terms (depending  on  $\eps^{-1}$)  in  the  diffusion \eqref{diffusion-X},  we  consider  the process
$({\widehat X}_t^\eps)_{t\ge 0} $(recall that $\disp
\overline{X}_t^\eps = X_t^\eps \slash \eps$) given by
\begin{align*}
\qs   t  \ge 0,  \qquad  {\widehat X}_t^\eps  =  X_t^\eps  + \eps (
\widehat b (\bar X_t^\eps)   - \widehat  b(\dfrac{x}{\eps}))
\end{align*}
Thanks to It\^o's formula (see \cite[Lemma
3.2]{Hai-Par}), we  have 
\begin{align}
\qs   t  \ge  0,\qquad  {\widehat X}_t^\eps  &= x +  \int_0^t \, (I
+ \dpa_x \hat b) \, c (\bar X_r^\eps) \, d r  + \int_0^t \, \L (\bar
X_r^\eps) \, d B_r. \label{rel1}
\end{align}
As  a  consequence  we deduce  that  the sequence
of processes $\{X_s^\eps,  \; 0\le s \le t,  \; \,  0<\eps \le
1\}$ is tight in   the space  $ \mathcal{C} ([0, t], \,\R^d)$ endowed with the
topology of  uniform  convergence. Moreover thanks to  the
martingale central limit theorem \cite[Theorem 7.1.4]{Eth-Kur}  which  follows from  Corollary \ref{ergodique}, we
have
$$
\int_0^\cdot  \L\,  (\bar  X_s^\eps) \, d B_s  \;
\Longrightarrow \; A^{1/2} \, B_\cdot \quad \text{ in
}\,\mathcal{C}([0,\, T];\R^d) \quad \text{ as }\,   \; \eps  \to  0.
$$
Hence   there   exists  a  subsequence   (still  note as   the
whole sequence)  such  that
$$ (X^\eps,  Y^\eps, M^\eps)  \Longrightarrow (X, Y, M)  \quad \text{ in }\,{\cal D}([0,T];\R^{2d+1}).$$

\noindent Let us   assume  that   the following extension of Corollary \ref{ergodique}
holds (the proof  is given   in  section  \ref{section-proofs}).
\begin{propo}\label{theoerg}
Let $\Psi:\R^d\times\R^N\rightarrow\R$ be a measurable function,
periodic with respect to its first variable, satisfying: \\
i) for any $R>0$, we can find  $K_R>0$ such that whenever
$(x,v,v')\in \R^{d}\times\R^N\times\R^N$, $|v|\leq R$ and $|v'|\leq
R$,  then we have $|\Psi(x,v)- \Psi(x,v')|\leq K_R|v-v'|$.\\
ii) there
exists $M>0$ such that for any $x\in\R^d, \;  v  \in  \R^N,
\quad   |\Psi(x, v)|\leq M  (1 + |v|)$. \\
Suppose additionally that $(V^\varepsilon)_{\varepsilon>0}$ is a
family of $\R^N$-valued processes, which is tight in ${\cal D}([0,T];\R^N)$
equipped  with   the $S$--topology  of Jakubowski and satisfies
\begin{equation}\label{hyp-unif}
\sup_{\varepsilon>0}\E(\sup_{0  \leq s \leq \t^\eps} e^{\l s}|V^\varepsilon_s|^2)<\infty.
\end{equation}
Then the following convergence holds: for any $\nu<\lambda$,
\begin{equation}\label{convergo}
 \E\Big|\int_{0}^{\t^\eps}e^{\nu r}\Psi(\bar{X}^\varepsilon_r,V^{\varepsilon}_r)\,dr- \int_{0}^{\t^\eps}e^{\nu r}\bar{\Psi}
 (V_r^\eps)\,dr\Big|\xrightarrow \; 0,\quad
  \text{ as }\varepsilon \text{ tends to }0,
\end{equation}
where $\disp\bar{\Psi}(v)=\int_{\T}\Psi(x,v)\,\mu(dx) $.
\end{propo}

From  now  on our strategy  consists in proving that the difference $Y^{\eps, x}_0-u(x)$ tends  to  0  as  $\eps$ goes  to 0.
However  in  the   following computations, we  are   faced  with the  lack  of smoothness
of  the   function  $u$.  To  overcome  the difficulty,  we approximate  the function  $u$ with  the
help  of  the  smooth  sequence  $(u^n)_{n\in\N}$ defined  in  Proposition \ref{regul-u}. Thus  we   consider,  for  every
$n \in  \N$
 the pair of processes $(\tl
Y_s^{\eps, n}  ,  \;  \tl Z_s^{\eps, n}  )_{s\ge 0}$ defined
by
\begin{align*}
\qs   s  \ge  0,  \qquad   {\tl Y}_s^{\eps,n}    =
Y_s^{\eps,x}     - u^n(\widehat X_{s\wedge\t^\eps}^\eps);  \qquad
{\tl Z}_s^{\eps,n} = Z_s^{\eps,x} -   \dpa_x u^n(\widehat
X_{s}^\eps)\L(\overline{X}_s^\eps).
\end{align*}

Our main result, Theorem \ref{main-result}, is an immediate consequence of

\begin{theo1}\label{yborne} The  following  holds \\
(i)  There   exists a constant  $C_{\ref{yborne}}   >  $0  such that
for every $\eps
>0,    n  \in \N$   and   $t \ge  0$,  we  have
$$ |{\tl Y}_t^{\eps,n}    | \,  \le  C_{\ref{yborne}} \quad   a.s.$$
(ii)
For  all  $\eta  >  0$  there  exists  an  integer   $n(\eta)$
such that   for all   $n  \ge   n(\eta)$,
\begin{equation*}
\limsup_{\eps  \td  0} |\tl Y_0^{\eps, n}| \le \eta. 
\end{equation*}
\end{theo1}
\begin{proof}
{\sc Step 1 : Proof of} (i).
It\^o's   formula   applied to the function $(t,y)
\mapsto e^{\l t}y^2$ yields that  for  every  $t \ge 0$
\begin{align*}
e^{\l t\wedge\tau^\eps} |Y_{t\wedge\tau^\eps}^{\eps,x}|^2 + \int_{t\wedge \t^\eps}^{\t^\eps} \, \,
e^{\l r} \, |Z_r^{\eps, x}|^2 \, d r = e^{\l \t^\eps}
|g(X_{\t^\eps}^\eps)|^2    - \int_{t\wedge \t^\eps}^{\t^\eps} \, \,
\l\,e^{\l r} \, |Y_r^{\eps, x}|^2 \, d r \\
+2 \int_{t\wedge \t^\eps}^{\t^\eps} \, \, e^{\l r} \, Y_r^{\eps, x}
\, f(\Th(\eps, r)) d r - 2 \int_{t\wedge \t^\eps}^{\t^\eps} \, \,
e^{\l r} \, Y_r^{\eps, x} \, Z_r^{\eps, x} \,d B_r,
\end{align*}
where   for  every  $\eps>0$  and   $r>0,  \; \Th(\eps, r)  = (\bar
X_r^{\eps, x}, X_r^{\eps, x}, Y_r^{\eps, x}, \, Z_r^{\eps, x})$. \\
Using  standard  estimates and assumptions  (H2.2),   we  have
$$ 2  e^{\l r} \, Y_r^{\eps, x}
\, f(\Th(\eps, r))   \le    e^{\l r}(K^2  +  2 \mu) |Y_r^{\eps,x}|^2  +
 e^{\l r} \,|Z_r^{\eps,x}|^2  +   \a e^{\l r}  |Y_r^{\eps,x}|^2  +   \frac{1}{\a} e^{\l r}  |f(\overline{X}_r^{\eps,x}, X_r^{\eps,x},0, 0)|^2
$$
with   $\a    =    \l  -  (2\mu  +  K^2) >0.$
Since    $g$  is bounded, there  exists    a  positive   constant   $K^\prime$   such  that for  every   $t\ge 0$,
\begin{equation*}
e^{\l t\wedge\tau^\eps} |Y_{t\wedge\tau^\eps}^{\eps,x}|^2  \le   K^\prime e^{\l \t^\eps}  + \frac{1}{\a} \int_{t\wedge
\t^\eps}^{\t^\eps} \, e^{\l r}  |f(\overline{X}_r^{\eps,x}, X_r^{\eps,x},0, 0)|^2 \,
d r  - 2 \int_{t\wedge \t^\eps}^{\t^\eps} \, \, e^{\l r} \,
Y_r^{\eps, x} \, Z_r^{\eps, x} \,d B_r
\end{equation*}
Taking   the  conditional  expectation
$\E^{{\cal F}_{t}} $, we deduce  thanks    to  assumption  (H2.2),
\begin{equation*}
\qs  \eps>0,  \;  \;  t \ge 0, \quad  |Y_t^{\eps,x}|^2  \le K^\prime \E^{{\cal F}_{t}} (e^{\l(\t^\eps  - t)^+})
\end{equation*}
From  the Markov  property,   the  term   of   the   right  hand side    is    equal  to   $\E_{X_t^\eps} \, (e^{\l\t^\eps})$. Hence from  \eqref{cond-exp}  there   exists   a    constant   $C_{\ref{yborne}}>0$  such   that
\begin{equation*}
\qs  \eps>0,  \;  \;  t \ge 0, \quad |Y_t^{\eps,x}|^2  \le  C_{\ref{yborne}}.
 \end{equation*}
(i)  follows,    since  $u^n$ is  bounded,   uniformly w. r. t. $n\ge1$. \\

\noindent{\sc Step 2 : an upper bound for $\tilde{Y}^{\eps,n}_0$.}
 Note that for every $s\ge 0$,
\begin{equation*} \widehat X_{s\wedge\t^\eps} = \widehat
X_{\t^\eps}^\eps    - \int_{s\wedge\t^\eps}^{\t^\eps} (I +
\dpa_x\widehat b)c (\overline{X}_s^\eps) d r
-\int_{s\wedge\t^\eps}^{\t^\eps}\L(\overline{X}_r^\eps)  d B_r.
\end{equation*} 
Since  $u^n   \in   C^\infty(\R^d)$,   It\^o's  Formula    yields    for  any   $s \ge  0$,
\begin{align*}
u^n(\widehat X_{s\wedge\t^\eps}^\eps)  &=u^n(\widehat
X_{\t^\eps}^\eps) - \int_{s\wedge\t^\eps}^{\t^\eps} \hat L^{\eps, n}(r) d
r - \int_{s\wedge\t^\eps}^{\t^\eps} \dpa_x u^n(\widehat X_{r}^\eps)
\L(\overline{X}_r^\eps)  d B_r,
\end{align*}
where   we   define    $$  \hat L^{\eps, n}(s) = \frac 1 2 \sum_{i, j}
(\L \L^*) (\overline{X}_s^\eps)\dpa_{x_i x_j} u^n(\widehat
X_{s}^\eps)    + \sum_i [(I +  \dpa_x \widehat b)c
(\overline{X}_s^\eps)]\dpa_x u^n(\widehat X_{s}^\eps). $$
Putting  pieces   together,  we deduce  that for any   $s \ge  0$,
\begin{equation*}
{\tl Y}_{s\wedge\tau^\eps}^{\eps,n}  =    g(X_{\t^\eps}^\eps)  - u^n(\widehat
X_{\t^\eps}^\eps) + \int_{s\wedge\t^\eps}^{\t^\eps} [f(\Th(\eps, r))
+ \hat L^{\eps, n}(r)]   d r   - \int_{s\wedge\t^\eps}^{\t^\eps}{\tl
Z}_r^{\eps,n} d B_r
\end{equation*}
Let  $ \nu  = 2\mu + K^2$. It\^o's  formula applied  to  the function
 $(t,y) \mapsto e^{\nu t}y^2$ yields   for   $t \ge   0$,
\begin{align*}
e^{\nu t \wedge \t^\eps} |{\tl Y}_{t \wedge\tau^\eps}^{\eps,n}|^2   +
\int_{t\wedge\t^\eps}^{\t^\eps}e^{\nu s} |{\tl
Z}_s^{\eps,n}|^2 d s &=e^{\nu\t^\eps} |g(X_{\t^\eps}^\eps)
- u^n(\widehat
X_{\t^\eps}^\eps)|^2   -  \int_{t\wedge\t^\eps}^{\t^\eps} \nu  \, e^{\nu s} |{\tl Y}_s^{\eps,n}|^2 d s  \\
&+2\int_{t\wedge\t^\eps}^{\t^\eps}e^{\nu s} \,\tl
Y_s^{\eps,n} \,[\hat L^{\eps, n}(s)  + f(\Th(\eps, s)]  d s
\\
&-2\int_{t\wedge\t^\eps}^{\t^\eps}e^{\nu s}\,  \tl
Y_s^{\eps,n} \, \tl Z_s^{\eps,n}   d B_s.
\end{align*}
Consider  the  decomposition
 \begin{align*} \tl
Y_s^{\eps,n} \,[\hat L^{\eps, n}(s)  + f(\Th(\eps, s)] &= \tl
Y_s^{\eps,n} \,[\hat L^{\eps, n}(s) + f(\bar X_s^{\eps, x}, X_s^{\eps,
x}, u^n (\widehat X_{s\wedge \t^\eps}^\eps), \dpa_x u^n(\widehat
X_{s}^\eps)
\L(\overline{X}_s^{\eps})] \\
&+ \tl Y_s^{\eps,n} \,[f(\Th(\eps, s) -  f(\bar X_s^{\eps, x},
X_s^{\eps, x}, u^n (\widehat X_{s\wedge \t^\eps}^\eps), \dpa_x
u^n(\widehat X_{s}^\eps) \L(\overline{X}_s^{\eps})] .
\end{align*} 
From assumption (H2.3), the second   term   of  the last right
 hand  side   is  less   than
 $$ \mu|\tl
Y_s^{\eps,n}|^2 + K |\tl Y_s^{\eps,n}| |\tl Z_s^{\eps,n}|.   $$
Hence  using  standard    estimates, we   deduce that
\begin{align*}
 |{\tl Y}_0^{\eps,n}|^2  &\le
\E(e^{\nu \t^\eps} |g(X_{\t^\eps}^\eps)- u^n(\widehat
X_{\t^\eps}^\eps)|^2)  + 2 \E
\int_{0}^{\t^\eps}e^{\nu s} \tl Y_s^{\eps,n}
[\d_{1, n}(\eps,   s)    +  \d_{2, n}(\eps,   s) ] d s \\ &
+ 2 \E
\int_{0}^{\t^\eps}e^{\nu s} \tl Y_s^{\eps,n}
[\hat L^{\eps, n}(s) -  \overline{L}  u^n (\widehat X_{s}^\eps)] ds
\end{align*}
with
\begin{align*}
\d_{1, n}(\eps,   s)    &=   f(\bar X_s^{\eps, x}, X_s^{\eps, x},
u^n (\widehat X_{s}^\eps), \dpa_x u^n(\widehat
X_{s}^\eps) \L(\overline{X}_s^{\eps}))  -  \bar f(X_s^{\eps, x}, u^n
(\widehat X_{s}^\eps), \dpa_x u^n(\widehat
X_{s}^\eps))    \\
\d_{2, n}(\eps,   s)    &=   \overline{L}   u^n (\widehat X_{s}^\eps)\,  + \bar f(X_s^{\eps,
x}, u^n (\widehat X_{s}^\eps), \dpa_x u^n(\widehat
X_{s}^\eps)).
\end{align*}
Let
\begin{align*}
C_n(\eps) &= \E(e^{\nu \t^\eps} |g(X_{\t^\eps}^\eps) -
u^n(\widehat X_{\t^\eps}^\eps)|^2)  + 2 \E
\int_{0}^{\t^\eps}e^{\nu s} \tl Y_s^{\eps,n}
[\hat L^{\eps, n}(s) -  \overline{L}  u^n (\widehat X_{s}^\eps)] ds \\
& + 2 \E
\int_{0}^{\t^\eps}e^{\nu s} \tl Y_s^{\eps,n}
\d_{1, n}(\eps,   s) ds  + 2 \E
\int_{0}^{\t^\eps}e^{\nu s} \tl Y_s^{\eps,n}
\d_{2, n}(\eps,   s) ds  \\
&:=  C_n(1,  \eps)   +  C_n(2, \eps) + C_n(3, \eps)  +  C_n(4, \eps) .
 \end{align*}
We have  that   for every $n \ge 1$ and
$\eps> 0$,
\begin{align}\label{basic-estim}
|\tl Y_0^{\eps,n}|^2 \le C_n(\eps).
\end{align}

\noindent{\sc Step 3 : Estimate of $C_n(1,\eps)$, $C_n(2,\eps)$ and $C_n(3,\eps)$.}
Assume w. l. o. g. that   $\nu$ and $\lambda$ have  the same sign.
From H\"older's inequality,  we deduce  that with $p=\lambda/\nu$, $q^{-1}+p^{-1}=1$,
\begin{equation}\label{estimC1}
0\le C_n(1, \eps) \le \Big(\E(e^{\l \t^\eps})\Big)^{1/p}
\Big(\E (|g(X_{\t^\eps}^\eps) - u^n(\widehat
X_{\t^\eps}^\eps)|^{2q})\Big)^{1/q}
\end{equation}
Fix  $n \in  \N$.  Thanks   to  the   tightness  of   $(X_s^\eps,
{\hat X}_s^\eps, \tl Y_s^{\eps,n})$, we deduce  from   Proposition 
\ref{theoerg}  that  for each $n\ge1$ fixed, $ C_n(2, \eps)  \xrightarrow[]{\eps \to 0} \;
0$  and $ C_n(3, \eps)  \xrightarrow[]{\eps \to 0} \;
0.$  \\

\noindent{\sc Step 4 : Estimate of $C_n(4, \eps)$.}
We shall take advantage of   Proposition \ref{regul-u}. For   this  end  for  any   $\d>0$ we  consider  a   function  $\f_\d  \in  \mathcal{C}(\R^d, [0, 1]) $  satisfying
$$
\f_\d(x) =
\begin{cases}
     1, & x  \in  G\setminus G^\d, \\
0,& \text{in } G^{2\d}.
\end{cases}
$$
We   have for   every  $\eps>0$ and  $n\ge1$,
\begin{align*}
C_n(4, \eps)
&=
2 \E\int_{0}^{\t^\eps}e^{\nu s} \tl Y_s^{\eps,n}
\d_{1, n}(\eps,   s) (1    -   \f_\d (\widehat X_{s}^\eps))ds  +
2\E\int_{0}^{\t^\eps}e^{\nu s} \tl Y_s^{\eps,n}
\d_{1, n}(\eps,   s) \f_\d (\widehat X_{s}^\eps)ds\\
&=C_n(4.1,\eps,\delta)+C_n(4.2,\eps,\delta).
\end{align*}
For  any    $\d>0$, all  the  arguments    in  the  integral  defining $C_n(4.1,\eps,\delta)$   are bounded, uniformly w. r. t. $\eps>0$.   So   using   Proposition \ref{regul-u},    we  deduce    that  thanks   to   Lebesgue's  theorem, uniformly w. r. t.  $\eps>0$,
\begin{equation}\label{estimC31}
\E\int_{0}^{\t^\eps}e^{\nu s} \tl Y_s^{\eps,n}\d_{1, n}(\eps,   s) (1    -   \f_\d (\widehat X_{s}^\eps))ds  
 \to   0,  \;  \text{as} \;  \;    n  \to  \infty.
\end{equation}
Moreover  we   have
$$ C_n(4.2,\eps,\delta)\le \disp  C_{\ref{yborne}}
\E\int_{0}^{\t^\eps}e^{\nu s} |\d_{1, n}(\eps,   s) | \f_\d (\widehat X_{s}^\eps)ds.$$
Using  Proposition \ref{hairer-Pard}, we  deduce  that for every   $t\ge 0$,
$$ \E
\int_{0}^{\t^\eps}  e^{\nu s} |
\d_{1, n}(\eps,   s)| \f_\d (\widehat{X}_{s}^\eps)ds  \xrightarrow{\eps \to  0}
\E\int_{0}^{\t} e^{\nu s}  |F_n(X_s)| \f_\d (X_{s})ds$$
with  $F_n(X_s) = \bar{L}   u^n (X_{s})\,  + \bar f(X_s, u^n (X_s), \dpa_x u^n(X_{s})),  \;  \; n\ge 1$.
Let us   prove that  
\begin{equation}\label{convC32}
 \lim_{\d \to  0}\sup_{n\ge1}\E\int_{0}^{\t} e^{\nu s}  |F_n(X_s)| \f_\d (X_{s})ds=0.
 \end{equation}
   To   this end   we  consider    two   cases :   \\

\noindent\textbf{Case  1:}  $\l < 0$. This   implies   $\nu   < 0$. Thanks  to   assumption  (H1.6) and Krylov's  estimate (see  \cite[Theorem 2.2]{Kry})   there   exists  an  integer   $p>d$  and    a    constant  $\tl K >0$  depending  on $d, p, \nu$  and  the   diameter    of  the    region $G$  such that
\begin{align*}
\E\int_{0}^{\t} e^{\nu s} \Big|F_n(X_s)\Big| \f_\d (X_{s})ds \le  \tl   K \Big(\int_{(0, \infty)\times G} e^{p \nu s} |F_n(x)|^p  \f_\d(x) d x d s\Big)^{1/p}
\end{align*}
 Then  H\"older's  inequality implies
 \begin{align*}
\sup_{n\ge1}\E\int_{0}^{\t} e^{\nu s} |F_n(X_s)| \f_\d (X_{s})ds \le  \tl   K \Big(\int_{(0, \infty)\times G} e^{p \nu s} \f_\d(x) d x d s\Big)^{1/2p} \Big(\sup_{n\ge1}\int_{(0, \infty)\times G} e^{p \nu s} |F_n(x)|^{2p}  d x d s\Big)^{1/2p}.
\end{align*}
Since   $\nu  < 0$   and     $F_n(\cdot)$ is bounded in $L^{2p}(G)$  (thanks  to  Proposition \ref{regul-u} (iv)),  the    last term    of  the  right   hand  side     is   finite.  Moreover  the   sequence    $(\f_\d(x))_{\d>0}$   is  decreasing  to   0 at any point $x\in G$,   as  $\d\downarrow 0$. As a  consequence    we   have   $$  \lim_{\d \to  0} \int_{(0, \infty)\times G} e^{p \nu s} \f_\d(x) d x d s  =    0.$$  \eqref{convC32} follows in case 1.  \\

\noindent\textbf{Case  2:}  $\l>0$. Let     $0  \vee \nu   <  \g   < \l$. Using   H\"older's  inequality     with  $p = \l/\g > 1$  and    $q$  s.t.   $p^{-1}  +   q^{-1}  =  1$,    we   have
\begin{align*}
 \E\int_{0}^{\t} e^{\nu s} |F_n(X_s)| \f_\d (X_{s})ds &=\E\int_{0}^{\t} e^{\g s} \f_\d (X_{s}) e^{(\nu  - \g) s} |F_n(X_s)| ds \\
 &\le  \Big(\E \int_0^\t e^{\l s} \f_\d (X_{s})ds\Big)^{1/p} \times
  \Big(\E \int_0^\t e^{q(\nu  - \g) s}  |F_n(X_s)|^q ds\Big)^{1/q}
  \end{align*}
Using again Krylov's  estimate,  we   deduce  that    (for   some   $m>d$  and   $\tl K>0$  depending    on  the  region $G,
q, \nu-\gamma$  and  $d$)
 $$\sup_{n\ge1}\E \int_0^\t e^{q(\nu  - \g) s}  |F_n(X_s)\Big|^q ds \le  \tl K  \Big(\sup_{n\ge1}\int_{(0,  \infty)\times G} e^{mq(\nu  - \g) s}  |F_n(x)\Big|^{mq}  dx d s\Big)^{1/m} < \infty,$$
 thanks  to   Proposition \ref{regul-u} (iv) and    $\nu  -  \g  < 0.$
 Moreover  by uniform   integrability,   we   have
 $$  \lim_{\d \to  0}  \E \int_0^\t e^{\l s} \f_\d (X_{s})ds   = 0.$$
  \eqref{convC32} follows in case 2. \\
   
   \noindent{\sc Step 5 : Proof of } (ii).
 It follows from  \eqref{convC32} that we can first choose $\delta(\eta)>0$ small enough, such that 
 \begin{equation}\label{be1}
 \sup_{n\ge1}\lim_{\eps\to0}C_n(4.2,\eps,\delta(\eta))\le \eta^2/3.
 \end{equation}
 From Proposition \ref{regul-u} (ii) and \eqref{estimC31}, we can next
  choose $n(\eta)$ large enough, such that  (with $p=\lambda/\nu$) both
 \begin{align}
 \sup_{x\in\partial G}|u^{n(\eta)}-g|(x)&\le\left(\sup_{\eps>0, x\in\overline{G}}\E_x(e^{\lambda\tau^\eps})\right)^{-1/2p}
\sqrt{ \eta^2/3},
 \nonumber
 \\ \label{be2}
 \sup_{\eps>0}C_{n(\eta)}(4.1,\eps,\delta(\eta))&\le\eta^2/3.
 \end{align}
 We now deduce from the above estimates, in particular \eqref{estimC1}, that
 \begin{align}\label{be3}
 \limsup_{\eps\to0}C_{n(\eta)}(1,\eps)&\le \eta^2/3,\\
 C_{n(\eta)}(2,\eps)+C_{n(\eta)}(3,\eps)&\to0,\quad\text{as }\ \eps\to0.
 \label{be4}
 \end{align}
(ii) now follows from \eqref{basic-estim}, \eqref{be1}, \eqref{be2}, \eqref{be3} and \eqref{be4}.
\end{proof}
\section{Proofs of Propositions and Corollary \ref{regul-u} to \ref{theoerg}}\label{section-proofs}
\subsection{Proof  of Proposition \ref{regul-u}}
To begin    with  let  us  establish  some   preliminaries    result. Since     $g \in \mathcal{C}^2(\dpa G)$  and  $\dpa G$  is  of   class  ${\cal  C}^2$ there   exists   a  function     $v \in  W^{2, p}(G)$   such  that $v(x)  =    g(x), \; x  \in  \dpa G.$\\
  Putting
$\Psi (x,r, q)  =  - \bar  Lv(x)  - \bar  f(x, v(x) + r,  \dpa_x v  + q),  \;  (x,  r,  q) \in  G\times  \R \times \R^d$,
we  have  the  following   consequence  of   \cite[Theorem 2]{Boc-Mur-Pue}.
 \begin{lem1} \label{lemme-u}Assume  that  (H1) and  (H2)  are   in force. Then  the  PDE
\begin{eqnarray}
\label{equation-ubar}
\begin{cases}
 \overline{L}   u(x)  +   \Psi(x, u(x), \dpa u(x))   = 0,  \quad x  \in  G,     \\
u (x) =  0,  \quad   x \in \dpa  G.
\end{cases}
\end{eqnarray}
admits at   least one  solution   $    u  \in   H_0^1(G)$.
\end{lem1}
\noindent  As  a  consequence,   we  have
\begin{propo}\label{solution-u}
The unique      viscosity  solution   $u$   of   equation   \eqref{lim-elliptic3}   satisfies
$$ u    \in  W^{2, p}(G),  \quad    \;  \text{for any}\; \;  p \ge 1.$$
\end{propo}
\begin{proof}
Let  us consider   $\bar   u  =      u  -    v$.
It   suffices   to   prove that
$\bar  u   \in W^{2, p}(G)  \cap   W_0^{1, p}(G) $,   $p\ge1$.
Let $\bar v \in  H_0^1(G)$  be  a solution of  \eqref{equation-ubar}.  It follows from $(H2.2)$   that
$\Psi(\cdot,\overline{v}(\cdot),\partial\overline{v}(\cdot))   \in  L^2(G)$. Consequently, from \cite{ADN},  $\bar   v \in W^{2, 2}(G)  \cap   W_0^{1, 2}(G)$.  We   then deduce that  $ \dpa_x \bar v  \in   H^1(G)  \hookrightarrow L^p(G)$ and   $\bar v \in L^p(G)$ for all   $p  \ge 2$
if $d\le2$, for $p=2d(d-2)^{-1}$ otherwise, by Sobolev embedding.
From  the  linear  growth   of   $\overline{f}$, we  deduce  that      $\Psi(\cdot,\overline{v}(\cdot),\partial\overline{v}(\cdot))   \in  L^p(G)$.  This   implies, again by \cite{ADN},
$$\bar  v    \in  W^{2, p}(G)  \quad   \text{and} \quad  \dpa_x  \bar   v   \in W^{1, p}(G). $$
 Using  again  the  Sobolev   embedding,  we  deduce   that $\bar  v$   and    $\dpa_x  \bar   v$  belong  to  $L^q(G)$   for    all $q$ if $d\le 4$, for $q=2d(d-4)^{-1}$ otherwise. Iterating this argument
 $\lceil \frac{d}{2}-1\rceil$
 times, we deduce that   $\bar   v \in W^{2, p}(G)$   for  all  $p\ge 1$.
Thus  $\bar  v$ is   a    viscosity solution   of      \eqref{equation-ubar}.  Now uniqueness of the  viscosity solution  of  the   elliptic  PDE \eqref{equation-ubar}  (see   \cite[Corollary 6.96]{Par-Rascanu})  implies   $\bar u   =  \bar  v$.  The result     follows since  $v \in W^{2, p}(G)$ for all $ p\ge1$.
\end{proof}
We  are   now  in  position   to   prove  Proposition \ref{regul-u}. To  this    purpose,    we  first    extend  the  function  $u$  as  an element  of  $ W^{2, p}(\R^d)$, which  is   possible  given  the  regularity  of  $\dpa G$  and $u$.
Let  $ \r :\R^d
\to \R$ be a smooth mollifier  with compact support, and define for
  $n  \ge 1,  \;    \r_n(x) = n^d \r(nx)$. We regularize $u$, the  solution  of  \eqref{lim-elliptic3},  by convolution : $u^n$  defined   as
    $u^n(x)    =   (u * \r_n)(x)$.
Thanks  to   Proposition  \ref{solution-u},     $u\in\mathcal{C}^1(\overline{G})$. This    implies that $(\dpa_x u^n)_{n\ge 1}$   is   uniformly  bounded and
\begin{equation}
\label{convergence-un}
(u^n,  \, \dpa_x u^n) \to    (u,\,  \dpa_x u) ,   \quad   as   \;  \;   n  \to   \infty,  \text{uniformly in} \;  \overline{G}.
\end{equation}
(i) and (ii) are established.
Let  us   prove   (iii).   
Since    $\overline{L}  u^n     =   (\overline{L}  u) * \r_n$,  we    have  
\begin{align*}
 \overline{L}  u^n(x)    +    \overline{f}(x,   u^n(x), \dpa_x u^n(x))   &= \Big[\overline{L}  u   +     \overline{f}(\cdot, u(\cdot),  \dpa_x(\cdot))\Big]*\r_n(x)   \\
&+  \overline{f}(x,   u^n(x), \dpa_x u^n(x))     -    \Big[\overline{f}(\cdot, u(\cdot),  \dpa_x(\cdot))\Big]*\r_n(x) \\
&:=A^n(x)      +   B^n(x)
\end{align*}
For   all  $x\in  G^\d,  \;    A^n(x)  =   0$  if $n$ is large  enough    such  that      $\text{Supp}(\r_n)  \subset B(0, \d)$, the  unit  ball   of    radius   $\d$. Moreover   since  $\overline{f}$ is  continuous   with  respect  to  its  second  and    third  arguments and  \eqref{convergence-un},  we     have $B^n(x)    \to    0$  as   $n\to   \infty$,  uniformly  w. r. t. $x$. 

Finally (iv) follows from the fact that, since $u$ has been extended to an element of $W^{2,p}(\R^d)$,
$u^n$ is bounded in $W^{2,p}(G)$.
\subsection{Proof   of Proposition \ref{lem:meyer}.}
From \eqref{Y-eps}, we deduce thanks to It\^o's formula
applied to the function  $(t,y)  \mapsto  e^{\l t}y^2$
\begin{align}\label{ito-meyer}
e^{\l t\wedge\tau^\eps} |Y_{t\wedge\tau^\eps}^{\eps,x}|^2 + \int_{t\wedge \t^\eps}^{\t^\eps} \, \,
e^{\l r} \, |Z_r^{\eps, x}|^2 \, d r = e^{\l \t^\eps}
|g(X_{\t^\eps}^\eps)|^2    - \int_{t\wedge \t^\eps}^{\t^\eps} \, \,
\l\,
e^{\l r} \, |Y_r^{\eps, x}|^2 \, d r \nonumber\\
+2 \int_{t\wedge \t^\eps}^{\t^\eps} \, \, e^{\l r} \, Y_r^{\eps, x}
\, f(\Th(\eps, r)) d r - 2 \int_{t\wedge \t^\eps}^{\t^\eps} \, \,
e^{\l r} \, Y_r^{\eps, x} \, Z_r^{\eps, x} \,d B_r.
\end{align}
Let   $0< \g<1$  and  $\b>0$ such    that   $\l =  2  \mu   +  \frac{K^2}{\g}  + \b$. Using    standard estimates, we  have
$$ 2 Y_r^{\eps, x}
\, f(\Th(\eps, r)) \;    \le   (2\mu  +  \frac{K^2}{\g}  +
\b)|Y_r^{\eps, x}|^2  +     \g  |Z_r^{\eps, x}|^2   + \frac{1}{\b}
|f(X_r^{\eps, x}, \overline{X}_r^{\eps, x}, 0, 0)|^2
$$
Hence   we deduce   from  \eqref{ito-meyer}
\begin{align}
e^{\l t\wedge\tau^\eps} |Y_{t\wedge\tau^\eps}^{\eps,x}|^2 + (1-\g) \int_{t\wedge \t^\eps}^{\t^\eps} \, \,
 e^{\l r}    |Z_r^{\eps, x}|^2
\, d r &\le K^2 \, e^{\l \t^\eps} + \frac{1}{\b}\int_{t\wedge
\t^\eps}^{\t^\eps} \, \, e^{\l r} \,
|f(X_r^{\eps, x}, \overline{X}_r^{\eps, x}, 0, 0)|^2 d r \nonumber\\
&- 2 \int_{t\wedge \t^\eps}^{\t^\eps} \, \, e^{\l r} \, Y_r^{\eps,
x} \, Z_r^{\eps, x} \,d B_r.\label{estim-Y}
\end{align}
It follows from Burkholder--Davis--Gundy's inequality that  there    exists   a    constant  $ C_{\ref{lem:meyer}} >0 $  such  that
$$
\E\Big[\sup_{0\le t\le \tau^\eps}e^{\l t}|Y_t^\eps|^2\, + \int_0^{\t^\eps} e^{\l
s}|Z_s^\eps|^2\, d
 s\Big]   \le   C_{\ref{lem:meyer}}\E\Big(K^2 \,e^{\l
\t^\eps} + \int_{0}^{\t^\eps} \, \, e^{\l r} \, |f(X_r^{\eps, x},
\overline{X}_r^{\eps, x}, 0, 0)|^2 d r \Big)
$$
which  is   enough  to   get  the desired result  thanks  to asumption (H2.2)  and  \eqref{cond-exp}. \hfill $\Box$
\subsection{Proof  of Corollary \ref{tightness}.}
Let   $T>0$. Note that  the  process $M_t^\eps  =  \disp \int_0^{t
\wedge \tau^\eps} \, Z_r^{\eps, x} \, d B_r,   0\le t \le T $
satisfies $CV_T^0 (M^\eps)  = 0.$ Using   standard estimates and   Doob's   inequality,   we     have  for    any  $\eps>0$,
$$
\sup_{0\le t \le T} \,\E  |M_t^\eps| \le\frac{1}{2} (1 + \E(\sup_{0\le t \le T} |M_t^\eps|^2))  \le\frac{1}{2} (1+ 4 \sup_{\eps>0} \, \E < M^\eps >_T)  $$
Since   for every
$0\le s  \le  T, \;  e^{\l s }(1 \vee  e^{- \l T})\ge 1$, we deduce    that
$$\sup_{\eps>0} \sup_{0\le t \le T} \,\E  |M_t^\eps|
\le   \frac{1}{2} (1  +  4 \sup_\eps \, \E \bigg( (1 \vee  e^{- \l
T}) \int_0^{T\wedge\tau^\eps} \, e^{\l s }|Z^{\eps,x}_s|^2 ds \bigg)   <  \infty.
$$
Consequently,   the sequence  $\{M_t^\eps, \; 0\le t \le T\}$  is tight in  ${\cal D} ([0, T],  \R^d)$. \\
Futhermore  let $0 = t_0 <   t_1 <... < t_n = T$ be a partition of
$[0, T]$, thanks to the assumptions on $f$ and standard estimates,
we have for all $\eps > 0$,
\begin{align*}
\E  \sum_{i = 0}^{n - 1}  \,  |\E (Y_{t_{i+1}}^\eps   - Y_{t_i}^\eps
|  {\cal F}_{t_i})|   &\le
 \E \int_0^{T\wedge \t^\eps}  |f(\overline{X}_s^\eps,\; X_s^\eps, \; Y_s^\eps, Z_s^\eps) |\, d s \\
&\le K  \;  \E (1 \vee  e^{- \l T})\int_0^{T\wedge \t^\eps} e^{\l s
}(1  +  |Y_s^\eps| +
| Z_s^\eps|) \, d s \\
&\le  K T  \,(1 \vee  e^{- \l T})  (1 +  \E \,   \sup_{0 \le s \le
T\wedge\t^\eps} \, e^{\l s } |Y_s^\eps|^2  + \E \int_0^{T\wedge\t^\eps} e^{\l s }
|Z_s^\eps|^2\, d s)
\end{align*}
From this and Proposition \ref{lem:meyer}, we deduce  that
 $$ \sup_{\eps>0} \Big(\sup_{0\le t  \le T\wedge\tau^\eps} \E|Y_t^\eps|  +  CV_T^0 (Y^\eps{\bf1}_{[0,\tau^\eps]})\Big) <\infty$$  which implies
  that  the sequence  $(Y_s^\eps)_{\eps > 0} $ satisfies the Meyer--Zheng tightness  criteria. \hfill $\Box$
\subsection{Proof  of Proposition \ref{theoerg}.}
Let  us   define   $\tl \Psi(x, y)    = \Psi(x, y)    -   \bar \Psi(y)$.
For every  $T>0$,  and   $\d>0$,   we   have  for any  $\nu < \l$,
$$ \Prb\left(\left|\int_{0}^{\t^\eps}e^{\nu r}\tl  \Psi(\bar{X}^\varepsilon_r,V^{\varepsilon}_r)\,dr\right| >\d\right)  \le   \Prb(\t^\eps > T)  +
\Prb\left(\left|\int_{0}^{T\wedge\t^\eps}e^{\nu r}\tl
\Psi(\bar{X}^\varepsilon_r,V^{\varepsilon}_r)\,dr\right| >\d\right) $$
Since  $\t^\eps  \Td  \t$,   the   sequence  $(\t^\eps,  \eps>0)$  is  tight  and   we  can  choose  $T$   large  enough  such  that  $\sup_{\eps>0} \Prb(\t^\eps > T) $  is arbitrary  small.  Moreover  the
second  term  of  the hand  right side  is  equal  to
$$\Prb\left(\left|\int_{0}^{T}e^{\nu r}\tl
\Psi(\bar{X}^\varepsilon_r,V^{\varepsilon}_r)\times 1_{\{0\le r \le
\t^\eps\}}\,dr\right|>\d\right).$$
Thanks     to   the     tightness   of  the  process $1_{\{0\le r \le \t^\eps\}}$
in  ${\cal D}([0,\infty))$,   \cite[Lemma 4.2]{Par1}
implies   that $$ \Prb\left(\left|\int_{0}^{T}e^{\nu r}\tl
\Psi(\bar{X}^\varepsilon_r,V^{\varepsilon}_r)\times 1_{\{0\le r \le
\t^\eps\}}\,dr\right|>\d\right) \to     0,\ \text{    as   }\eps   \to  0.$$
It remains to prove that the collection of random  variables
$\disp\left\{\int_{0}^{\t^\eps}e^{\nu r}\tl  \Psi(\bar{X}^\varepsilon_r,V^{\varepsilon}_r)\,dr,\ \eps>0\right\}$
is uniformly integrable. Since $|\tl\Psi(x,v)|\le 2M(1+|v|)$,
\begin{align}
\left|\int_{0}^{\tau^\eps}e^{\nu r}\tl\Psi(\bar{X}^\eps_r,V^\eps_r)\,dr\right|
\le2M \int_{0}^{\tau^\eps}e^{\nu r}(1+|V^\eps_r|)dr
\label{inegalite-erg}
\end{align}
Now   we  consider   two   cases : \\

\noindent \textbf{Case  1:}  $\l>0$.  We consider the case $\nu>0$ only, from which the result follows for $\nu\le0$.
\eqref{inegalite-erg} implies
\begin{align*}
\left|\int_{0}^{\tau^\eps}e^{\nu r}\tl\Psi(\bar{X}^\eps_r,V^\eps_r)\,dr\right|
&\le \frac{2M}{\nu}\left(e^{\nu\tau^\eps}+ 2e^{\nu\tau^\eps/2}\sup_{0\le r\le \tau^\eps}
e^{\nu r/2}|V^\eps_r|\right)
\end{align*}
The collection of random  variables $\{\xi_\eps:=e^{\nu\tau^\eps},\ \eps>0\}$ is tight since $\sup_{\eps>0}\E[|\xi_\eps|^p]<\infty$ with $p=\lambda/\nu>1$, by \eqref{cond-exp}. Now choose $2<p<2\lambda/\nu$, and let $q$ be such that $q^{-1}+p^{-1}=1$. Then from Young's inequality,
$$e^{\nu\tau^\eps/2}\sup_{0\le r\le \tau^\eps}e^{\nu r/2}|V^\eps_r|\le\frac{1}{p}
e^{p\nu\tau^\eps/2}+\frac{1}{q}\sup_{0\le r\le \tau^\eps}e^{q\nu r/2}|V^\eps_r|^q.$$
Tightness of the last right hand side follows from $p\nu/2<\lambda$ and $q\nu/2<\lambda$, since $q<2$,
\eqref{cond-exp} and \eqref{hyp-unif}. \\

\noindent \textbf{Case  2:}  $\l<0$.  Then $\nu<0$ as well.   We  have  for    any  $\eps>0$,
\begin{align*}
 \int_{0}^{\tau^\eps}e^{\nu r}(1+|V^\eps_r|)dr &=  \int_{0}^{\tau^\eps}e^{(\nu  -  \l) r} e^{\l r}(1+|V^\eps_r|)dr \\
 &\le  \frac{1}{\l - \nu}(1  +   \sup_{0\le r\le \t^\eps}e^{\l r}  |V^\eps_r|)
\end{align*}
Since  $\l<0$, we   have   $\E[(\sup_{0\le r\le \t^\eps}e^{\l r}  |V^\eps_r|)^2] \le  \E[\sup_{0\le r\le \t^\eps}e^{\l r}  |V^\eps_r|^2]$.  Hence using   standard  estimates  and \eqref{hyp-unif}, we  deduce  that    the  last     right  hand side  is    tight. \hfill $\Box$

\noindent\textbf{Acknowledgements }: The authors wish to thank Fran\c cois Hamel and Emmanuel Russ for valuable discussions concerning the proof of Proposition \ref{regul-u}, and Alessio Porretta for pointing out to us the reference  \cite{Boc-Mur-Pue}. \\
The work of the second author has been supported by LATP at
Universit\'e de Provence, Marseille and AIRES-Sud, a program from
the French Ministry of Foreign and European Affairs implemented by
the Institut de Recherche pour le D\'eveloppement (IRD--DSF).

\end{document}